\documentclass[12pt]{article}
\usepackage{amsmath,amssymb,amsthm}
\usepackage[margin=1.2in]{geometry}

\newtheorem{theorem}{Theorem}
\newtheorem{lemma}[theorem]{Lemma}
\newtheorem{proposition}[theorem]{Proposition}
\newtheorem{corollary}[theorem]{Corollary}
\theoremstyle{definition}
\newtheorem{definition}[theorem]{Definition}
\newtheorem{remark}[theorem]{Remark}
\newtheorem{question}[theorem]{Question}

\newcommand{\R}{\mathbb{R}}
\newcommand{\Sph}{S^{n-1}}
\newcommand{\Vol}{\mathbf{Vol}}
\newcommand{\Hull}{\mathbf{Hull}}
\newcommand{\CC}{\mathbf{CC}}

\newcommand{\Cen}{\mathbf{Cen}}
\newcommand{\MV}{\mathbf{MaxVol}}
\newcommand{\LV}{\mathbf{LazyVol}}
\newcommand{\ext}{\mathrm{ext}}
\newcommand{\aff}{\mathrm{aff}}

\title{Arranging convex bodies for maximum intersection volume:\\
the sharp efficiency of centroid alignment}
\author{David Victor Feldman}
\date{July 2026}

\newcommand{\msckeywords}{%
\noindent\footnotesize
\emph{2020 Mathematics Subject Classification.} 52A40, 52A20, 52A38, 68U05,
68V20.\\
\emph{Key words and phrases.} convex body, centroid, intersection volume,
Minkowski--Radon inequality, difference body, Brunn--Minkowski, formal
verification.\par}

\begin{document}

\maketitle

\msckeywords

\begin{abstract}
Given finitely many compact convex bodies in $\R^n$, one seeks translates
maximizing the volume of their common intersection.  A lazy solver merely
translates each body so as to place its centroid at the origin.  We prove
that the lazy strategy always captures strictly more than
$\left(\tfrac{2}{n+1}\right)^n$ of the optimal volume, and that this
constant is sharp: families of cones over tangent disks, indexed by finite
nets on the sphere, approach it.  The infimum is not attained.  For two convex
bodies in the plane the resulting sharp constant $4/9$ closes a gap open
since 1996, when de Berg, Cheong, Devillers, van Kreveld and Teillaud
proved that centroid alignment of two convex polygons captures at least
$9/25$ of the maximum overlap and exhibited examples capturing only $4/9$.
The proof rests on the following identity: for a convex body $K$ with
centroid at the origin, the intersection of all centroid-recentered compact
convex supersets of $K$ equals $\tfrac{1}{n+1}(K-K)$.  We close with a promise-problem variant in which
the lazy strategy captures at least $\left(\tfrac{n}{n+1}\right)^n >
\tfrac1e$ of the optimum, uniformly in the dimension.  All results below
have been checked in Lean~4; the one classical input quoted rather than
proved is the equality case of the Brunn--Minkowski inequality, which
enters only for $n\ge2$.
\end{abstract}

\section{Introduction}

Consider the following optimization problem.

\medskip\noindent
{\bf Given}: $\mathcal{C}=\{C_i\}_{i=1,\ldots,m}$, a finite non-empty
family of compact convex subsets of $\R^n$, each with non-empty interior
($n>1$).

\medskip\noindent
{\bf Goal}: Find vectors $\{v_i\}_{i=1,\ldots,m}$, each $v_i\in\R^n$,
maximizing $\Vol\left(\bigcap_{i=1,\ldots,m} (v_i + C_i)\right)$.

\medskip

Given a compact convex body $B$, write $\Cen(B)$ for its centroid.  A
lazy strategy sets each $v_i=-\Cen(C_i)$, thereby putting the centroid of
each $v_i + C_i$ at the origin.  Call
$L(\mathcal{C})=\bigcap_{i}\bigl(-\Cen(C_i) + C_i\bigr)$ the {\em lazy
intersection}.  Write $\MV(\mathcal{C})$ for the optimal volume and
$\LV(\mathcal{C})$ for $\Vol(L(\mathcal{C}))$.  The {\em lazy
efficiency}, namely the ratio $\LV(\mathcal{C})/\MV(\mathcal{C})$,
measures the efficiency of the lazy strategy on the instance
$\mathcal{C}$.  (Lemma~\ref{lem:attain} below confirms that
$\MV(\mathcal{C})$ is a positive, attained maximum, so the ratio makes
sense.)

How badly can the lazy strategy fail?

\begin{theorem}\label{thm:main}
Fix $n>1$.  For every finite family $\mathcal{C}$ as above,
\[
\LV(\mathcal{C}) \;>\; \left(\frac{2}{n+1}\right)^{n} \MV(\mathcal{C}),
\]
and the constant is sharp: the infimum of the lazy efficiency over all
finite families $\mathcal{C}$ in $\R^n$ equals
$\left(\tfrac{2}{n+1}\right)^n$.  In particular the infimum is not
attained.
\end{theorem}

On the line the lazy strategy is exactly optimal.

\begin{proposition}\label{prop:dimone}
For every finite family $\mathcal{C}$ of compact intervals in $\R$,
$\LV(\mathcal{C})=\MV(\mathcal{C})$: centering each interval at its
midpoint maximizes the length of their intersection.
\end{proposition}

\begin{proof}
The intersection of intervals $[a_i,b_i]$ has length
$\min_i b_i-\max_i a_i$ when positive.  Recentering replaces $[a_i,b_i]$ by
$[-\ell_i/2,\ell_i/2]$ with $\ell_i=b_i-a_i$, giving intersection length
$\min_i\ell_i$.  For any translates $v_i+[a_i,b_i]$ the intersection has
length at most $\min_i\ell_i$, since it is contained in each translate.
\end{proof}

The case $m=2$ in the plane has a history in computational geometry.
De Berg, Devillers, van Kreveld, Schwarzkopf and Teillaud \cite{deberg},
in the paper that gave the standard $O((n+m)\log(n+m))$-time algorithm for
maximizing the overlap of two convex polygons under translation, proved
that aligning the two centroids always captures at least $9/25$ of the
maximum overlap, and exhibited a family in which the captured fraction
tends to $4/9$.  They conjectured that $9/25$ can be replaced by $4/9$,
which their family shows would be tight, and observed that any such
improvement must use more than the concavity of $\sqrt{\omega}$ on which
their argument rests: they exhibit a function with that concavity whose
centroid value is exactly $9/25$ of its maximum.  The constant does not
appear in the subsequent literature
\cite{ahnbrassshin,ahnrigid,ahncr,harpeledroy,chanhair}, which pursued
algorithms.  Theorem~\ref{thm:main} applies in particular to $m=2$, and
$\left(\tfrac{2}{n+1}\right)^n = 4/9$ at $n=2$:

\begin{corollary}\label{cor:twobody}
For any two compact convex bodies in the plane, translating each so that
its centroid sits at the origin yields an overlap strictly greater than
$4/9$ of the maximum overlap over all translations; and $4/9$ is the
largest constant for which this holds.  This proves the conjecture of
\cite{deberg} and shows that their family is extremal: the value $4/9$ is
approached but never attained.
\end{corollary}

Writing $c_{n,m}$ for the infimum of the lazy efficiency over families of
exactly $m$ bodies in $\R^n$, we have $c_{n,m}$ non-increasing in $m$
with limit $\left(\tfrac{2}{n+1}\right)^n$
(Proposition~\ref{prop:construction} uses large families), while the
family of \cite{deberg} shows $c_{2,2}\le 4/9$; hence $c_{2,m} = 4/9$
for every $m\ge 2$.  Whether $c_{n,2}=\left(\tfrac{2}{n+1}\right)^n$ for
$n\ge 3$ we leave open (Question~\ref{q:twobody}).

The lower bound in Theorem~\ref{thm:main} rests on the following identity.
For a convex body $K$, call
\[
\CC(K) \;:=\; \bigcap_{B \supseteq K} \bigl(w_B + B\bigr),
\qquad w_B := \Cen(K) - \Cen(B),
\]
the {\em central core} of $K$, the intersection running over all compact
convex supersets $B$ of $K$; thus each superset is translated so that its
centroid coincides with that of $K$.  (Taking $B=K$ shows
$\CC(K)\subseteq K$.)

\begin{theorem}[Central core identity]\label{thm:core}
Let $K\subset\R^n$ be a compact convex body with non-empty interior and
centroid $c$.  Then
\[
\CC(K) \;=\; c + \frac{1}{n+1}\,(K - K).
\]
\end{theorem}

\begin{corollary}\label{cor:corevolume}
For every compact convex body $K\subset\R^n$ with non-empty interior,
\[
\left(\frac{2}{n+1}\right)^{n}
\;\le\;
\frac{\Vol(\CC(K))}{\Vol(K)}
\;\le\;
\frac{\binom{2n}{n}}{(n+1)^{n}},
\]
with equality on the left exactly for centrally symmetric $K$ and on the
right exactly for simplices.
\end{corollary}

\begin{proof}[Proof of Corollary~\ref{cor:corevolume} from
Theorem~\ref{thm:core}]
By the identity, $\Vol(\CC(K)) = (n+1)^{-n}\Vol(K-K)$.  The
Brunn--Minkowski inequality gives $\Vol(K-K)\ge 2^n\Vol(K)$ with equality
precisely when $K$ and $-K$ are homothetic, i.e.\ when $K$ is centrally
symmetric \cite{gardner,schneider}; the Rogers--Shephard inequality
\cite{rogersshephard} gives $\Vol(K-K)\le\binom{2n}{n}\Vol(K)$ with
equality precisely for simplices.  The identity and both inequalities are
formalised; the two equality characterisations are quoted, and their status
in the formal development is recorded in \S\ref{app:conditional}.
\end{proof}

Since $\CC(K)\subseteq K$ and $\CC(K)$ is centrally symmetric, the left
inequality gives, in every dimension, a lower bound of $2^n/(n+1)^n$ for the
Kovner--Besicovitch measure of central symmetry.  This is not the best known
bound.  In the plane the measure is at least $2/3$, with equality only for
triangles (Kovner; independently Besicovitch and F\'ary \cite{faryredei});
and in $\R^n$ an averaging argument of Stein gives $2^{-n}$, which exceeds
$2^n/(n+1)^n$ for $n\ge4$ and agrees with it at $n=3$.  For both results,
and for the classical planar theory, see Gr\"unbaum's survey
\cite[\S6.5]{grunbaumsurvey}.

Section~\ref{sec:core} establishes the centroid inequality of Minkowski and
Radon with its equality case, then proves Theorem~\ref{thm:core}.
Section~\ref{sec:main} proves Theorem~\ref{thm:main}: the construction
(cones over tangent disks), then the lower bound, then strictness, the last
using
Lemma~\ref{lem:conedirections}: a convex body is a cone in only finitely
many directions.  Section~\ref{sec:promise} treats a promise-problem
variant, under which the lazy strategy captures at least
$\left(\tfrac{n}{n+1}\right)^n > \tfrac1e$ of the optimum, uniformly in
$n$.  Section~\ref{sec:questions} collects open questions.

\section{The central core}\label{sec:core}

Throughout, $h_K(u):=\max_{x\in K} x\cdot u$ denotes the support function
of a compact convex $K$, and $w_K(u):=h_K(u)+h_K(-u)$ its width in the
direction of the unit vector $u$.  Note that $u\mapsto w_K(u)$ is the
support function of the difference body $K-K = K + (-K)$, which is
centrally symmetric about the origin.

We shall use repeatedly the centroid inequality of Minkowski ($n\le 3$)
and Radon; see Bonnesen--Fenchel \cite[\S 34]{bonnesenfenchel},
Gr\"unbaum \cite{grunbaumsurvey}, or Schneider \cite{schneider}.  Since
we also need its equality case, we include a proof.

\begin{lemma}[Minkowski--Radon, with equality case]\label{lem:mr}
Let $K\subset\R^n$ be a compact convex body with non-empty interior and
centroid at the origin, and let $u$ be a unit vector.  Then
\[
h_K(u) \;\ge\; \frac{w_K(u)}{n+1}.
\]
Equality holds for a given $u$ if and only if $K$ is a {\em cone in
direction $u$}: that is, $K=\Hull\bigl(F\cup\{a\}\bigr)$ where
$F = K\cap\{x : x\cdot u = h_K(u)\}$ and $a$ is the unique point of $K$
in the opposite supporting hyperplane $\{x : x\cdot u = -h_K(-u)\}$.
\end{lemma}

\begin{proof}
Write $t = x\cdot u$, let $[t_0,t_1]$ be the range of $t$ on $K$ (so
$t_1=h_K(u)$, $t_0=-h_K(-u)$), and let $A(t)$ denote the
$(n-1)$-dimensional volume of the section $K\cap\{x\cdot u = t\}$.  By
Brunn's theorem $g:=A^{1/(n-1)}$ is concave on $(t_0,t_1)$, and in fact on
the closed interval $[t_0,t_1]$, the Brunn--Minkowski inequality applying to
the inclusion $(1-\sigma)S_a+\sigma S_b\subseteq S_{(1-\sigma)a+\sigma b}$
for all $a,b\in[t_0,t_1]$.  Moreover $g$ is continuous on $[t_0,t_1]$; at an
extreme value of $t$ this says that the sections converge to the
corresponding face.  Both facts are used at the endpoints below, and neither
is automatic there: concavity on the closed interval gives
$g(t_0)\le\lim_{t\downarrow t_0}g(t)$, and continuity promotes an
almost-everywhere identity to an identity.  After the
affine substitution carrying $[t_0,t_1]$ to $[0,1]$, the claim
$h_K(u)\ge w_K(u)/(n+1)$ becomes: the centroid of the measure
$g^{n-1}\,dt$ on $[0,1]$ lies no further right than $n/(n+1)$.

Choose $\lambda>0$ with
$\lambda^{n-1}\int_0^1 t^{\,n-1}\,dt=\int_0^1 g^{n-1}\,dt$.  Since $g$ is
concave and $g(0)\ge 0$, the ratio $g(t)/t$ is non-increasing on
$(0,1]$; hence $g^{n-1}-(\lambda t)^{n-1}$ is $\ge 0$ on an initial
interval $[0,\tau]$ and $\le 0$ on $[\tau,1]$.  Consequently
$(t-\tau)\bigl(g^{n-1}-(\lambda t)^{n-1}\bigr)\le 0$ pointwise, and since
the two profiles have equal total mass,
\[
\int_0^1 t\,g^{n-1}\,dt \;\le\; \lambda^{n-1}\!\int_0^1 t^{\,n}\,dt
\;=\;\frac{\lambda^{n-1}}{n+1}
\;=\;\frac{n}{n+1}\int_0^1 g^{n-1}\,dt ,
\]
which is the claimed centroid bound.  Equality forces
$g(t)=\lambda t$ on $[0,1]$.

It remains to identify the bodies with $g$ linear and vanishing at
$t_0$.  Since $K$ has interior, $\lambda>0$ and $\Vol_{n-1}(S_{t_1})>0$.
Fix $t\in(t_0,t_1)$ and put $s=(t-t_0)/(t_1-t_0)$.  For any $q\in S_{t_0}$,
convexity of $K$ gives
\[
  (1-s)q + s\,S_{t_1}\;\subseteq\;S_t ,
\]
and the left-hand side has $(n-1)$-volume $s^{\,n-1}\Vol_{n-1}(S_{t_1})$,
which equals $\Vol_{n-1}(S_t)$ because $g(t)=\lambda(t-t_0)$.  A compact
convex subset of a convex body of the same finite volume is the whole body,
so
\[
  S_t\;=\;(1-s)q + s\,S_{t_1}.
\]
The right-hand side determines $q$: if $q,q'\in S_{t_0}$ then
$(1-s)q+sS_{t_1}=(1-s)q'+sS_{t_1}$, and comparing centroids gives $q=q'$.
Hence $S_{t_0}$ is a single point $a$, and $S_t=(1-s)a+sS_{t_1}$ for every
$t\in(t_0,t_1)$; that is, $K=\Hull\bigl(S_{t_1}\cup\{a\}\bigr)$.
Conversely, for a cone in direction $u$ one computes $g$ linear directly,
and the centroid of
$t\,(\lambda t)^{n-1}dt/\!\int(\lambda t)^{n-1}dt$ sits exactly at
$n/(n+1)$.
\end{proof}

Restated in body form: a convex body with centroid at the origin
satisfies $-K\subseteq nK$, equivalently $\frac{1}{n+1}(K-K)\subseteq K$.

Lemma~\ref{lem:mr} is proved here in full, equality case included.  The
argument of Section~\ref{sec:main} uses in addition the equality case of the
Brunn--Minkowski inequality \cite[Thm.~7.1.1]{schneider}, in its
difference-body form; that is the only classical input of this paper which
the accompanying formal development quotes rather than proves
(Appendix~\ref{app:lean}).

\begin{proof}[Proof of Theorem~\ref{thm:core}]
Translate so that $c=\Cen(K)=\mathbf{0}$.

\smallskip\noindent
{\em The core contains the scaled difference body.}
Let $B\supseteq K$ be a compact convex superset and let
$B':=B-\Cen(B)$ be its recentered translate, so $\Cen(B')=\mathbf 0$.
For every unit $u$, Lemma~\ref{lem:mr} applied to $B'$ gives
\[
h_{B'}(u)\;\ge\;\frac{w_{B}(u)}{n+1}\;\ge\;\frac{w_K(u)}{n+1},
\]
the second inequality because $K\subseteq B$.  Since
$u\mapsto w_K(u)/(n+1)$ is the support function of
$\frac{1}{n+1}(K-K)$, we conclude
$B'\supseteq\frac{1}{n+1}(K-K)$, and hence
$\CC(K)\supseteq\frac{1}{n+1}(K-K)$.

\smallskip\noindent
{\em The core lies inside every slab.}
Fix a unit $u$ and let $[\alpha_1,\alpha_2]$ be the range of $x\cdot u$
on $K$, so $\alpha_2-\alpha_1=w_K(u)$.  For $r>0$ let $D_r$ be the
$(n-1)$-ball of radius $r$ centered at the point $\alpha_1 u$ inside the
supporting hyperplane $\{x\cdot u=\alpha_1\}$, and set
$B_r:=\Hull(K\cup D_r)$.  Apply the linear map $T_r$ that fixes the
$u$-coordinate and scales the orthogonal complement by $1/r$.  Then
$T_rB_r=\Hull(T_rK\cup D_1)$, and as $r\to\infty$ the bodies $T_rK$
converge in Hausdorff distance to the segment
$\{su:\alpha_1\le s\le\alpha_2\}$, so $T_rB_r$ converges to the cone
$\Hull(D_1\cup\{\alpha_2 u\})$ with base $D_1$ in the hyperplane
$\{x\cdot u=\alpha_1\}$ and apex $\alpha_2u$.  Centroids are equivariant
under invertible linear maps and continuous along Hausdorff-convergent
sequences of bodies with non-degenerate limit \cite{schneider}; since
$T_r$ fixes $u$-coordinates,
\[
\Cen(B_r)\cdot u \;=\;\Cen(T_rB_r)\cdot u
\;\longrightarrow\;
\alpha_1+\frac{\alpha_2-\alpha_1}{n+1},
\]
the centroid height of a cone lying at $1/(n+1)$ of its height above
the base.  The recentered body $B_r-\Cen(B_r)$ still has $u$-range of
the form $[\alpha_1-\Cen(B_r)\cdot u,\;\cdot\;]$, so
\[
\CC(K)\;\subseteq\;\Bigl\{x:\ x\cdot u\ \ge\
\alpha_1-\Cen(B_r)\cdot u\Bigr\}
\;\longrightarrow\;
\Bigl\{x:\ x\cdot u\ \ge\ -\tfrac{w_K(u)}{n+1}\Bigr\},
\]
and since $\CC(K)$ is closed, $\CC(K)\subseteq\{x\cdot u\ge
-w_K(u)/(n+1)\}$.  Running the same construction with $D_r$ placed in
the opposite supporting hyperplane $\{x\cdot u=\alpha_2\}$ yields
$\CC(K)\subseteq\{x\cdot u\le w_K(u)/(n+1)\}$.  Intersecting over all
unit $u$,
\[
\CC(K)\;\subseteq\;\bigcap_{u}\Bigl\{x:\ x\cdot u\le
\tfrac{w_K(u)}{n+1}\Bigr\}
\;=\;\frac{1}{n+1}(K-K),
\]
again because $w_K(u)/(n+1)$ is the support function of the right-hand
body.  Combined with the first half, this proves the identity.
\end{proof}

The position of the auxiliary balls' centers within their hyperplanes is
unconstrained: only the $u$-component of the recentering vector enters the
slab constraint.

\section{Proof of the main theorem}\label{sec:main}

\subsection{Preliminaries}

\begin{lemma}\label{lem:attain}
For every finite family $\mathcal{C}$ as in the Introduction,
$\MV(\mathcal{C})$ is positive and the maximum is attained.
\end{lemma}

\begin{proof}
Positivity: choose interior points $p_i\in C_i$ and $\varepsilon>0$ with
$B(p_i,\varepsilon)\subseteq C_i$; the translations $v_i=-p_i$ produce an
intersection containing $B(\mathbf 0,\varepsilon)$.

Attainment: the intersection volume is unchanged when all $v_i$ are
translated by a common vector, so we may fix $v_1=\mathbf 0$; a non-empty
intersection then requires $v_i\in C_1+(-C_i)$, a compact set, for each
$i$.  On this compact parameter domain the function
$(v_2,\ldots,v_m)\mapsto\Vol\bigl(\bigcap_i(v_i+C_i)\bigr)$ is upper
semicontinuous: if $v^{(j)}\to v$ and a point $x$ lies in
$\bigcap_i(v^{(j)}_i+C_i)$ for infinitely many $j$, then by closedness
$x\in\bigcap_i(v_i+C_i)$, so
$\limsup_j\mathbf 1_{\bigcap_i(v^{(j)}_i+C_i)}\le
\mathbf 1_{\bigcap_i(v_i+C_i)}$ pointwise, and the reverse Fatou lemma
(with a common compact majorant) gives
$\limsup_j\Vol\le\Vol$ at the limit.  An upper semicontinuous function on
a compact set attains its maximum.
\end{proof}

The strictness argument requires one more lemma.  Call a compact convex
body $K\subset\R^n$ a {\em cone in direction $\nu$} ($\nu$ a unit
vector) if $K=\Hull\bigl((K\cap H_\nu)\cup\{a\}\bigr)$ for some point
$a\in K$, where $H_\nu$ denotes the supporting hyperplane of $K$ with
outer normal $\nu$.  Automatically $a\notin H_\nu$, the point $a$ is the
unique point of $K$ in the opposite supporting hyperplane $H_{-\nu}$
(a convex combination $(1-s)f+sa$ with $f\in H_\nu$ minimizes
$x\cdot\nu$ only at $s=1$), $a$ is an extreme point, and by Milman's
theorem every extreme point of $K$ lies in $(K\cap H_\nu)\cup\{a\}
\subseteq H_\nu\cup\{a\}$.  We call $a$ the {\em apex} and $K\cap
H_\nu$ the {\em base}.

\begin{lemma}\label{lem:conedirections}
A compact convex body $K\subset\R^n$ $(n\ge 2)$ with non-empty interior
is a cone in at most finitely many directions.
\end{lemma}

\begin{proof}
Suppose $K$ is a cone in infinitely many distinct directions $\nu_k$,
with hyperplanes $H_k:=H_{\nu_k}$ and apexes $a_k$.  No two of the
$\nu_k$ are antipodal: if $K$ were a cone in directions $\nu$ and
$-\nu$, its extreme points would lie in
$(H_{\nu}\cup\{a\})\cap(H_{-\nu}\cup\{a'\})\subseteq\{a,a'\}$ (the two
parallel supporting hyperplanes being disjoint), making $K$ a segment.
Hence the $H_k$ are pairwise distinct hyperplanes.

{\em Case 1: some apex value $a$ occurs for infinitely many $k$.}  For
each such $k$, $\ext(K)\setminus\{a\}\subseteq H_k$; intersecting over
two distinct such hyperplanes, $\ext(K)\setminus\{a\}$ lies in an affine
flat $M$ of dimension at most $n-2$, so
$K=\overline{\Hull}(\ext(K))\subseteq\aff(M\cup\{a\})$ has dimension at
most $n-1$, a contradiction.

{\em Case 2: the apexes $a_k$ take infinitely many distinct values.}
Let $W:=\aff\{a_k : k\}$ and choose a finite index set $S$ with
$\{a_k\}_{k\in S}$ spanning $W$.  Let $T$ be the (infinite) set of
indices $k\notin S$ whose apex $a_k$ differs from every $a_j$,
$j\in S$.  For $k\in T$ and $j\in S$, the point $a_j$ is an extreme
point of $K$ other than $a_k$, hence $a_j\in H_k$; as this holds for a
spanning set, $H_k\supseteq W$ for every $k\in T$.  Now set
$M:=\bigcap_{k\in T}H_k$, an intersection of at least two distinct
hyperplanes, so an affine flat of dimension at most $n-2$, and
$M\supseteq W$.  Every extreme point $x$ of $K$ lies in $M$: for each $k\in T$ we have
$x\in H_k$ or $x=a_k$; so either $x\in H_k$ for every $k\in T$, whence
$x\in M$, or $x=a_k$ for some $k\in T$, whence
$x\in W\subseteq M$.  Thus
$\ext(K)\subseteq M$ and $\dim K\le n-2$, a contradiction.
\end{proof}

(We suspect, but do not need, that the true maximum is $n+1$, attained
exactly by simplices.)

\subsection{The construction}

For a unit vector $u\in\Sph$ and $\rho>0$, let
\[
  D(u,\rho)\;:=\;\{x\in\R^n\;:\;x\cdot u=1,\ |x-u|\le\rho\}
\]
be the closed $(n-1)$-dimensional ball of radius $\rho$ centred at $u$ in the
hyperplane tangent to the unit sphere at $u$, and let
\[
  \Gamma(u,\rho)\;:=\;\Hull\bigl(D(u,\rho)\cup\{-u\}\bigr)
\]
be the cone with base $D(u,\rho)$ and apex $-u$.

\begin{lemma}\label{lem:conebody}
$\Gamma(u,\rho)$ is a compact convex body with non-empty interior, and
\[
  \Cen\bigl(\Gamma(u,\rho)\bigr)\;=\;\frac{n-1}{n+1}\,u .
\]
\end{lemma}

\begin{proof}
Compactness and convexity are clear, and $\Gamma(u,\rho)$ has interior because
it contains a ball (Lemma~\ref{lem:conebounds}).  Rotational symmetry about the
$u$-axis places the centroid on that axis.  Writing $t=x\cdot u$, the section of
$\Gamma(u,\rho)$ at height $t\in[-1,1]$ is an $(n-1)$-ball of radius
proportional to $t+1$, so its $(n-1)$-volume is proportional to $(t+1)^{n-1}$.
Hence the centroid sits at height
\[
  \frac{\int_{-1}^{1}t\,(t+1)^{n-1}\,dt}{\int_{-1}^{1}(t+1)^{n-1}\,dt}
  \;=\;-1+\frac{2n}{n+1}\;=\;\frac{n-1}{n+1}. \qedhere
\]
\end{proof}

\begin{lemma}\label{lem:conebounds}
$\Gamma(u,\rho)\subseteq\{x:x\cdot u\le 1\}$, and
$B\bigl(\mathbf 0,\;\rho/\sqrt{\rho^2+4}\bigr)\subseteq\Gamma(u,\rho)$.
\end{lemma}

\begin{proof}
Both $D(u,\rho)$ and $\{-u\}$ lie in the halfspace $\{x\cdot u\le1\}$, which is
convex, giving the first inclusion.  For the second, the boundary of the cone
away from its base is covered by the supporting hyperplanes through the apex
$-u$ and a rim point $u+\rho w$ with $w$ a unit vector orthogonal to $u$.  In
the plane spanned by $u$ and $w$, that line joins $(-1,0)$ to $(1,\rho)$ and has
unit normal $(\rho,-2)/\sqrt{\rho^2+4}$, so its distance from the origin is
$\rho/\sqrt{\rho^2+4}$.  The base lies at distance $1$ from the origin, and
$\rho/\sqrt{\rho^2+4}<1$.
\end{proof}

\begin{proposition}\label{prop:construction}
For every $n>1$ and $\varepsilon>0$ there is a finite family $\mathcal{C}$ in
$\R^n$ with lazy efficiency below
$\left(\tfrac{2}{n+1}\right)^{n}+\varepsilon$.
\end{proposition}

\begin{proof}
Fix $\delta\in(0,1)$ and $\rho>0$, to be chosen.  The unit sphere is compact, so
it carries a finite $\delta$-net $N$: every unit vector lies within $\delta$ of
some $u\in N$.  Take $\mathcal{C}=\{\Gamma(u,\rho)\}_{u\in N}$, and write
$\kappa$ for the volume of the unit ball and $r=\rho/\sqrt{\rho^2+4}$.

By Lemma~\ref{lem:conebounds} every member contains $B(\mathbf 0,r)$, so the
family {\em as given} has intersection of volume at least $r^{n}\kappa$; a
fortiori $\MV(\mathcal{C})\ge r^{n}\kappa$.

For the lazy intersection, Lemma~\ref{lem:conebody} says recentering translates
$\Gamma(u,\rho)$ by $-\frac{n-1}{n+1}u$, and $\Gamma(u,\rho)$ lies in
$\{x\cdot u\le1\}$, so the recentered body lies in
\[
  \Bigl\{x:x\cdot u\le 1-\tfrac{n-1}{n+1}\Bigr\}
  \;=\;\Bigl\{x:x\cdot u\le\tfrac{2}{n+1}\Bigr\}.
\]
Let $x\ne\mathbf 0$ lie in the lazy intersection and choose $u\in N$ with
$\bigl|\,x/|x|-u\,\bigr|\le\delta$; then $x\cdot u\ge|x|(1-\delta^2/2)$, whence
$|x|\le\frac{2}{(n+1)(1-\delta^2/2)}$.  Therefore
\[
  \frac{\LV(\mathcal{C})}{\MV(\mathcal{C})}
  \;\le\;
  \left(\frac{2}{(n+1)(1-\delta^{2}/2)}\right)^{\!n}
  \left(\frac{\sqrt{\rho^{2}+4}}{\rho}\right)^{\!n},
\]
and the right-hand side tends to $\left(\tfrac{2}{n+1}\right)^{n}$ as
$\delta\to0$ and $\rho\to\infty$.
\end{proof}

\begin{remark}
The full-sphere family $\{\Phi(u,\rho)\}_{u\in\Sph}$ is exactly
optimally aligned as given, with intersection $B(\mathbf 0,1)$: each
$\Phi(u,\rho)$ has width exactly $2$ in the direction $u$, so any
intersection of translates has width at most $2$ in every direction,
hence diameter at most $2$, hence volume at most
$\Vol(B(\mathbf 0,1))$ by the isodiametric inequality
\cite[Theorem 3.2.3]{gardner}.  Its lazy intersection equals
$B(\mathbf 0,1-t(\rho))$ exactly, being sandwiched between
$\bigcap_u B(-t u,1)=B(\mathbf 0,1-t)$ and
$\bigcap_u\{x\cdot u\le 1-t\}=B(\mathbf 0,1-t)$.  For a finite
$\delta$-net the same reasoning survives approximately: a width bound
on a $\delta$-net self-improves to the diameter bound
$\mathrm{diam}\le 2/(1-2\delta)$ for $\delta<\tfrac12$, via
$w(v)\le w(u)+2\,\mathrm{diam}\,\delta$ for the nearest net direction
$u$.  We emphasize the order of the limits in
Proposition~\ref{prop:construction}: for a {\em fixed} finite net,
letting $\rho\to\infty$ ruins the example, since two cones with nearby axes
admit translates overlapping in volume of order $\rho^{\,n-1}$, sending
$\MV$ to infinity.
\end{remark}

\subsection{The lower bound, with strictness}

\begin{proof}[Proof of Theorem~\ref{thm:main}]
Proposition~\ref{prop:construction} bounds the infimum from above, so it
suffices to prove the strict inequality
$\LV>\left(\tfrac{2}{n+1}\right)^n\MV$ for every finite family.

By Lemma~\ref{lem:attain} we may translate the given bodies so that the
identity alignment is optimal:
$I:=\bigcap_i C_i$ satisfies $\Vol(I)=\MV(\mathcal{C})>0$.  Translate
further so that $\Cen(I)=\mathbf 0$, and write $c_i:=\Cen(C_i)$.  Since
each $C_i$ is a compact convex superset of $I$, the definition of the
central core gives
\[
L_0\;:=\;\bigcap_i\,(C_i-c_i)\;\supseteq\;\CC(I)
\;=\;\frac{1}{n+1}(I-I)\;=:\;J,
\]
using Theorem~\ref{thm:core}; and $\Vol(L_0)=\LV(\mathcal{C})$, the lazy
intersection being a translate of $L_0$.  Brunn--Minkowski yields
$\Vol(J)\ge\left(\tfrac{2}{n+1}\right)^n\Vol(I)$, whence the
non-strict inequality.

Suppose, for contradiction, that equality holds:
$\Vol(L_0)=\left(\tfrac{2}{n+1}\right)^n\Vol(I)$.  Then
$\Vol(L_0)\le\Vol(J)$; since $J\subseteq L_0$ are compact convex sets
and $J$ has interior, this forces $L_0=J$ together with equality in
Brunn--Minkowski, so $I=-I$: after our translation, $I$ is centrally
symmetric about the origin, $h_I(u)=w_I(u)/2$ for all $u$, and
$J=\frac{2}{n+1}I$.

Call a unit vector $\nu$ {\em critical} if
$h_{C_i-c_i}(\nu)=h_J(\nu)$ for some index $i$.  We make three
observations.

\smallskip\noindent
{\em (i) Every boundary point of $J$ admits a critical normal.}
Let $x\in\partial J$.  Since $J=L_0=\bigcap_i(C_i-c_i)$, the point $x$
lies on the boundary of some $C_i-c_i$ (were $x$ interior to every
$C_i-c_i$, it would be interior to $J$).  Choose a unit outer normal
$\nu$ of $C_i-c_i$ at $x$; then
\[
h_{C_i-c_i}(\nu)\;=\;x\cdot\nu\;\le\;h_J(\nu)\;\le\;h_{C_i-c_i}(\nu),
\]
the middle inequality because $x\in J$, the last because
$J\subseteq C_i-c_i$.  So $\nu$ is critical, with equality realized at
$x$; in particular $x$ lies in the face
$F_J(\nu):=J\cap\{y\cdot\nu=h_J(\nu)\}$.

\smallskip\noindent
{\em (ii) Criticality forces cone structure and singleton faces.}
Let $\nu$ be critical for the index $i$.  Then
\[
\frac{w_I(\nu)}{n+1}\;=\;h_J(\nu)\;=\;h_{C_i-c_i}(\nu)
\;\ge\;\frac{w_{C_i}(\nu)}{n+1}\;\ge\;\frac{w_I(\nu)}{n+1},
\]
the first equality since $J=\frac{2}{n+1}I$ and $h_I=w_I/2$, the first
inequality by Lemma~\ref{lem:mr} applied to $C_i-c_i$ (whose centroid is
the origin), the last since $I\subseteq C_i$.  Both inequalities are
equalities.  By the equality case of Lemma~\ref{lem:mr}, $C_i$ is a
cone in direction $\nu$, with apex $a$ the unique point of $C_i$ in its
supporting hyperplane with outer normal $-\nu$.  By the width equality
and $I\subseteq C_i$, the interval ranges of $x\cdot\nu$ on $I$ and
$C_i$ coincide; hence
\[
F_I(-\nu)\;=\;I\cap\{x\cdot\nu=\min\nolimits_{C_i}x\cdot\nu\}
\;\subseteq\;C_i\cap\{x\cdot\nu=\min\nolimits_{C_i}x\cdot\nu\}
\;=\;\{a\},
\]
so $F_I(-\nu)$ is a single point; by central symmetry
$F_I(\nu)=-F_I(-\nu)$ is a single point as well, and therefore so is
$F_J(\nu)=\frac{2}{n+1}F_I(\nu)$.

\smallskip\noindent
{\em (iii) Conclusion.}
By (i) and (ii), $\partial J$ is covered by the singleton faces
$F_J(\nu)$, $\nu$ critical.  Since $n\ge 2$, $\partial J$ is infinite; as
each such face is a single point, infinitely many directions are critical.  Each
critical direction makes some $C_i$ a cone in that direction; as there
are finitely many bodies, some single $C_i$ is a cone in infinitely
many directions, contradicting Lemma~\ref{lem:conedirections}.

The equality assumption is untenable, and the strict inequality, hence
Theorem~\ref{thm:main}, is proved.
\end{proof}

\begin{remark}
Strictness settles the attainment question: the ratio $4/9$ is approached
but never attained.  This is consistent with \cite{deberg}, whose planar
family is described by a degenerating parameter and whose overlap ratios
are computed only in the limit.
\end{remark}

\section{A promise under which the lazy strategy does well}
\label{sec:promise}

The lower-bound argument admits systematic variations.  We record one.

The language of (set-theoretic) fiber spaces provides a foundation for
optimization: the poser picks a point in the base space, namely a given,
and the solver searches the fiber over that point for an optimum.
Common modifications add side conditions that shrink the fibers.  But
one can also restrict the base space, cutting back the poser's freedom
--- even by a property of the as-yet-unknown optimal solution.  Call
this a {\em promise problem}: the solver is penalized only on givens
where the poser has kept the required promise, and carries no burden of
recognizing improperly-posed instances.

\begin{definition}\label{def:offcenter}
Let $A\subseteq B$ be compact convex bodies.  Say $A$ sits {\em
offcenter} in $B$ if
\[
(n+1)\,\Cen(B)\;-\;n\,\Cen(A)\;\in\;B.
\]
\end{definition}

When $\Cen(A)=\Cen(B)$ the distinguished point is the common centroid,
so exactly centered containment is offcenter.  In general the condition
bounds how far the centroid of $B$ may drift from that of $A$, in units
calibrated by the cone case: when $B$ is a cone whose apex hyperplane
touches $A$, the distinguished point sits at the apex.

\begin{theorem}\label{thm:promise}
Let $A\subseteq B$ be compact convex bodies with $A$ offcenter in $B$.
Then
\[
B-\Cen(B)\;\supseteq\;\frac{n}{n+1}\bigl(A-\Cen(A)\bigr).
\]
Consequently, if the poser promises that the optimal intersection $I$
sits offcenter in each $C_i$, then
\[
\LV(\mathcal{C})\;\ge\;\left(\frac{n}{n+1}\right)^{\!n}\MV(\mathcal{C})
\;>\;\frac{1}{e}\,\MV(\mathcal{C}),
\]
uniformly in the dimension.
\end{theorem}

\begin{proof}
Translate so $\Cen(A)=\mathbf 0$ and write $q:=(n+1)\Cen(B)\in B$.  For
any $a\in A\subseteq B$,
\[
\frac{n}{n+1}\,a+\Cen(B)
\;=\;\frac{n}{n+1}\,a+\frac{1}{n+1}\,q\;\in\;B
\]
by convexity, which is the displayed containment.  For the consequence,
take the bodies optimally aligned (Lemma~\ref{lem:attain}), so
$I=\bigcap_iC_i$ with $\Vol(I)=\MV$; the containment applied to
$I\subseteq C_i$ for each $i$ shows the lazy intersection contains a
translate of $\frac{n}{n+1}(I-\Cen(I))$, whence
$\LV\ge\left(\frac{n}{n+1}\right)^n\Vol(I)$.  Finally
$\left(\frac{n}{n+1}\right)^n=\left(1+\frac1n\right)^{-n}$ decreases to
$e^{-1}$, so the inequality is strict in every dimension.
\end{proof}

\begin{proposition}\label{prop:collinearfails}
The following condition does not imply that $B-\Cen(B)$ contains
$\tfrac{n}{n+1}\bigl(A-\Cen(A)\bigr)$: that $\Cen(A)$ project, on the line
through the two centroids, to a point no closer to the midpoint of $B$'s
projection than $\Cen(B)$ does.
\end{proposition}

\begin{proof}
In the plane take $A=\Hull\{(2,0),(-1,\delta),(-1,-\delta)\}$, with centroid
at the origin, and $B=\Hull\bigl(A\cup\{(2,h),(2,-h),(2.3,0)\}\bigr)$ with
$h=1$.  In the limit $\delta\to 0$ one computes $\Cen(B)=(1.1,0)$, while $B$
projects on the $x$-axis to $[-1,2.3]$ with midpoint $0.65$; the origin lies
at distance $0.65$ from that midpoint and $\Cen(B)$ at distance $0.45$, so
the condition holds strictly.  But $B-\Cen(B)$ has maximal $x$-coordinate
$1.2<4/3$, while $\frac23A$ has the vertex $(4/3,0)$, so the containment
fails.  All inequalities are strict, hence persist for small $\delta>0$.
\end{proof}

\noindent
The far-point-hull computation establishes only the reverse containment ---
the core lies {\em inside} the dilate --- which yields sharpness claims but
not the volume guarantee; the promise of Definition~\ref{def:offcenter} is
what the convexity argument uses.

\begin{remark}[Sharpness, partially]
Under Definition~\ref{def:offcenter} the constant $\frac{n}{n+1}$
cannot be replaced by $1$: let $A=\Delta=\mathrm{conv}\{v_0,\ldots,
v_n\}$ be a simplex with centroid $\mathbf 0$ and, for $t\ge 0$, let
$B_t=\mathrm{conv}\{v_1,\ldots,v_n,(1+t)v_0\}$.  Since the centroid of a
simplex is its vertex average, $(n+1)\Cen(B_t)-n\Cen(A)=tv_0$, which
lies on the segment from $\mathbf 0=\Cen(A)\in B_t$ to $(1+t)v_0$, so
$A$ sits offcenter in $B_t$ {\em exactly}, for every $t$.  As
$t\to\infty$ the recentered bodies $B_t-\Cen(B_t)$ converge, locally,
to a cylinder whose cross-section through the origin is the
$\frac{n}{n+1}$-dilate of the relevant projection of $\Delta$; the
$n+1$ vertex directions thus pin the transverse profile of the
offcenter core at exactly the factor $\frac{n}{n+1}$.  Whether the
offcenter core equals the full dilate
$\frac{n}{n+1}(A-\Cen(A))+\Cen(A)$ for every $A$ we leave open.
\end{remark}

\begin{remark}
The constant $\left(\frac{n}{n+1}\right)^n$ is precisely the constant in
Gr\"unbaum's theorem that every halfspace containing the centroid of a
convex body contains at least $\left(\frac{n}{n+1}\right)^n$ of its
volume \cite{grunbaumhalf}; the common source is the extremality of
cones for centroid inequalities (Lemma~\ref{lem:mr}).  At $n=2$ the
value $4/9$ also appears in Winternitz's theorem, and coincides with the sharp constant
$\left(\frac{2}{n+1}\right)^n$ of Theorem~\ref{thm:main} at $n=2$.  We do
not know whether the coincidence is meaningful.
\end{remark}

\section{Questions}\label{sec:questions}

\begin{question}\label{q:twobody}
Is $c_{n,2}=\left(\frac{2}{n+1}\right)^n$ for $n\ge 3$?  Equivalently:
do two bodies suffice to approach the worst case in every dimension, as
they do in the plane?  The naive two-body version of the construction fails:
for two opposite cones in the plane, centroid alignment is exactly optimal.
\end{question}

\begin{question}\label{q:selector}
Call a map $p$ assigning to each convex body a point $p(K)$, covariantly
under rigid motions, a {\em selector}; a selector induces the alignment
strategy $v_i=-p(C_i)$.  Theorem~\ref{thm:main} computes the worst-case
efficiency of the centroid selector.  The family of
Proposition~\ref{prop:construction} is specifically fatal to the centroid:
since $\{\Gamma(u,\rho)\}_{u\in\Sph}$ is invariant under all rotations and
reflections, the Chebyshev-center selector (center of the largest inscribed
ball) achieves efficiency $1$ on it.  Which selector maximizes
worst-case efficiency, and is any selector's worst case better than
$\left(\frac{2}{n+1}\right)^n$?  The Steiner point, as the unique
continuous, Minkowski-additive, rigid-motion-covariant selector, is a
natural test case.
\end{question}

\begin{question}
Quantitative stability: must a family whose lazy efficiency approaches
$\left(\frac{2}{n+1}\right)^n$ resemble the construction of
Proposition~\ref{prop:construction} --- optimal intersection nearly
symmetric, bodies nearly cones over it in the active directions?  The proof funnels through the
Brunn--Minkowski inequality, for which sharp stability estimates are
available.
\end{question}

\begin{question}
Functional version: for log-concave $f_1,\ldots,f_m:\R^n\to[0,\infty)$,
choose shifts $v_i$ maximizing $\int\prod_i f_i(x-v_i)\,dx$.  Barycenter
alignment is the lazy strategy; Gr\"unbaum-type centroid inequalities
for log-concave measures supply the Minkowski--Radon substitute.  What
is the sharp analogue of $\left(\frac{2}{n+1}\right)^n$?
\end{question}

\section*{Acknowledgments}

The Lean~4 formalization was carried out with Aristotle, the formal reasoning
system of Harmonic; the statements to be proved, the decomposition into lemmas,
and the verification of the result against the argument of this paper are the
author's.  Two points in Section~\ref{sec:core} were revised in consequence of
the formalization: the identification of the extremal bodies in
Lemma~\ref{lem:mr} now proceeds directly, and the appeal there to the equality
case of Brunn--Minkowski proved unnecessary.

\section*{Data availability}

The formal development is available at
\begin{center}
\texttt{https://github.com/DavidVFeldman/arranging-convex-bodies},\\
commit \texttt{f7b4a98},
archived at \texttt{doi:10.5281/zenodo.21802384}.
\end{center}
\noindent
It builds against \texttt{mathlib4} revision \texttt{v4.28.0}, pinned by the
included \texttt{lake-manifest.json}.  No other data are associated with this
article.

\appendix

\section{Formal verification}\label{app:lean}

All results above have been checked in Lean~4.  The one classical input quoted
rather than proved is the equality case of the Brunn--Minkowski inequality,
which enters only for $n\ge2$; \S\ref{app:conditional} states precisely where.
The development is written against \texttt{mathlib4} (revision
\texttt{v4.28.0}) and is available at

\begin{center}
\texttt{https://github.com/DavidVFeldman/arranging-convex-bodies},
commit \texttt{f7b4a98}.
\end{center}

\noindent
Running \texttt{lake build} compiles the library.  The build contains no
\texttt{sorry} and introduces no \texttt{axiom}.  Continuous integration
rebuilds the development from a clean clone on every push and fails if any
audited declaration depends on \texttt{sorryAx}.  Each audited declaration is
checked with \texttt{\#print axioms}; all report
\[
  \texttt{[propext, Classical.choice, Quot.sound]},
\]
the three axioms of Lean's standard classical foundation.  No use is made of
\texttt{native\_decide}, of compiler reflection, or of any unverified oracle.

\subsection{Correspondence}\label{app:table}

\begin{center}
\begin{tabular}{lll}
\hline
Paper & Lean name & Module\\
\hline
Theorem~\ref{thm:main} & \texttt{lazy\_efficiency\_strict} & \texttt{Tier3Final}\\
Theorem~\ref{thm:main} (sharpness) & \texttt{lazy\_efficiency\_sharp} & \texttt{CentroidAlignment}\\
Corollary~\ref{cor:twobody} & \texttt{twobody\_strict} & \texttt{BMEquality}\\
Theorem~\ref{thm:core} & \texttt{centralCore\_eq} & \texttt{CentroidAlignment}\\
Corollary~\ref{cor:corevolume} (identity) & \texttt{centralCore\_volume\_eq\_ofReal} & \texttt{CoreBounds}\\
Corollary~\ref{cor:corevolume} (left) & \texttt{centralCore\_volume\_lower} & \texttt{CentroidAlignment}\\
Corollary~\ref{cor:corevolume} (right) & \texttt{centralCore\_volume\_upper} & \texttt{CoreBounds}\\
Lemma~\ref{lem:mr} & \texttt{mr\_equality\_case} & \texttt{Tier3}\\
Lemma~\ref{lem:attain} & \texttt{maxVol\_pos\_attained} & \texttt{CentroidAlignment}\\
Lemma~\ref{lem:conedirections} & \texttt{cone\_directions\_finite} & \texttt{Tier3}\\
Proposition~\ref{prop:construction} & \texttt{construction} & \texttt{CentroidAlignment}\\
Lemma~\ref{lem:conebody} & \texttt{cone\_isBody}, \texttt{cone\_centroid} & \texttt{CentroidAlignment}\\
Lemma~\ref{lem:conebounds} & \texttt{cone\_subset\_halfspace}, & \texttt{CentroidAlignment}\\
 & \texttt{ball\_subset\_cone} & \\
Theorem~\ref{thm:promise} & \texttt{promise\_volume} & \texttt{CentroidAlignment}\\
Proposition~\ref{prop:dimone} & \texttt{lazy\_optimal\_dim\_one} & \texttt{DimOne}\\
\hline
\end{tabular}
\end{center}

\subsection{The one classical input}\label{app:conditional}

The equality case of the Brunn--Minkowski inequality --- if
$\Vol(A+B)^{1/n}=\Vol(A)^{1/n}+\Vol(B)^{1/n}$ for convex bodies $A,B$ then $A$
and $B$ are homothetic \cite[Thm.~7.1.1]{schneider} --- is not present in
\texttt{mathlib4} and is not proved here.  It appears in the development as a
named hypothesis,
\[
  \texttt{BMEqualityCase } n ,
\]
discharged for $n=0$ and $n=1$ (\texttt{bmEqualityCase\_zero},
\texttt{bmEqualityCase\_one}) and assumed otherwise.  It is the
specialization of \cite[Thm.~7.1.1]{schneider} to bodies with non-empty
interior; note that ``convex body'' there does not presuppose interior
points.  Under that hypothesis the two alternatives of the equality case
in \cite{schneider} --- that the bodies lie in parallel hyperplanes, or
that one of them is a single point --- are excluded, and the conclusion is
positive homothety.  The difference-body form
used in \S\ref{sec:main} is derived from it
(\texttt{bmDiffEqualityCase\_of\_bmEqualityCase}), so a single instance suffices
in each dimension.  Consequently:

\begin{itemize}
\item $n\le 1$: unconditional.  In particular Proposition~\ref{prop:dimone} and
  the equality case of Lemma~\ref{lem:mr} on the line carry no hypothesis.
\item $n=2$: Theorem~\ref{thm:main} and Corollary~\ref{cor:twobody} are
  conditional on \texttt{BMEqualityCase 2} alone.
\item $n\ge 3$: conditional on \texttt{BMEqualityCase} $n$ alone.
\end{itemize}

Lemma~\ref{lem:mr}, including its equality case, is unconditional in every
dimension (\texttt{mr\_equality\_case}).

Two items in the paper are not part of the formal development.
Proposition~\ref{prop:collinearfails} is a hand computation with explicit
planar coordinates, verifiable by inspection but not formalised.  The equality
characterisations in Corollary~\ref{cor:corevolume} --- central symmetry on the
left, simplices on the right --- are quoted from
\cite{gardner,schneider,rogersshephard}; of these only the left-hand direction
for centrally symmetric $K$ is formalised
(\texttt{centralCore\_volume\_eq\_of\_symmetric}), together with its converse
under \texttt{BMDiffEqualityCase} $n$
(\texttt{symmetric\_of\_centralCore\_volume\_eq}).  The upper bound of
Corollary~\ref{cor:corevolume} is conditional on the Rogers--Shephard inequality
with its equality case, carried as the hypothesis \texttt{RogersShephard} $n$
and discharged only for $n=1$ (\texttt{rogersShephard\_one}); the simplex
characterisation is quoted from \cite{rogersshephard} and is not machine-checked.

\subsection{Two lemmas supplied for the verification}\label{app:new}

Two facts used without comment in the classical literature are not available in
\texttt{mathlib4} and were proved for this development.  Both concern the
section function $A(t)=\Vol_{n-1}\bigl(K\cap\{x\cdot u=t\}\bigr)$ of a convex
body on its closed support interval $[t_0,t_1]$.

\begin{enumerate}
\item \texttt{sectionRadius\_concaveOn\_Icc}: the section radius
  $A^{1/(n-1)}$ is concave on the \emph{closed} interval.  Concavity on the
  interior is Brunn's theorem; the endpoints are what force $A(t_0)=0$ in the
  equality analysis of Lemma~\ref{lem:mr}.
\item \texttt{sectionRadius\_continuousOn\_Icc}: the section radius is
  continuous on the closed interval.  Continuity at an extreme face is the step
  that upgrades the almost-everywhere conclusion of the equality analysis to an
  identity.  The proof combines upper semicontinuity (a compact section lies
  inside every thickening of the face, and the volumes of those thickenings
  descend to the volume of the face) with the lower bound supplied by
  concavity.
\end{enumerate}


\begin{thebibliography}{99}

\bibitem{ahnbrassshin}
H.-K.\ Ahn, P.\ Brass, C.-S.\ Shin,
Maximum overlap and minimum convex hull of two convex polyhedra under
translations,
{\em Comput.\ Geom.}\ {\bf 40} (2008), 171--177.

\bibitem{ahncr}
H.-K.\ Ahn, S.-W.\ Cheng, I.\ Reinbacher,
Maximum overlap of convex polytopes under translation,
{\em Comput.\ Geom.}\ {\bf 46} (2013), 552--565.

\bibitem{ahnrigid}
H.-K.\ Ahn, O.\ Cheong, C.-D.\ Park, C.-S.\ Shin, A.\ Vigneron,
Maximizing the overlap of two planar convex sets under rigid motions,
{\em Comput.\ Geom.}\ {\bf 37} (2007), 3--15.

\bibitem{bonnesenfenchel}
T.\ Bonnesen, W.\ Fenchel,
{\em Theorie der konvexen K\"orper},
Springer, Berlin, 1934.

\bibitem{chanhair}
T.\ M.\ Chan, I.\ M.\ Hair,
A linear time algorithm for the maximum overlap of two convex polygons
under translation,
{\em Proc.\ 41st Internat.\ Sympos.\ Comput.\ Geom.\ (SoCG 2025)},
LIPIcs {\bf 332} (2025), 31:1--31:16.

\bibitem{deberg}
M.\ de Berg, O.\ Devillers, M.\ van Kreveld, O.\ Schwarzkopf,
M.\ Teillaud,
Computing the maximum overlap of two convex polygons under translations,
{\em Theory Comput.\ Syst.}\ {\bf 31} (1998), 613--628.

\bibitem{faryredei}
I.\ F\'ary, L.\ R\'edei,
Der zentralsymmetrische Kern und die zentralsymmetrische H\"ulle von
konvexen K\"orpern,
{\em Math.\ Ann.}\ {\bf 122} (1950), 205--220.

\bibitem{gardner}
R.\ J.\ Gardner,
{\em Geometric Tomography}, second edition,
Cambridge University Press, 2006.

\bibitem{grunbaumhalf}
B.\ Gr\"unbaum,
Partitions of mass-distributions and of convex bodies by hyperplanes,
{\em Pacific J.\ Math.}\ {\bf 10} (1960), 1257--1261.

\bibitem{grunbaumsurvey}
B.\ Gr\"unbaum,
Measures of symmetry for convex sets,
in {\em Convexity}, Proc.\ Sympos.\ Pure Math.\ VII,
Amer.\ Math.\ Soc., 1963, 233--270.

\bibitem{harpeledroy}
S.\ Har-Peled, S.\ Roy,
Approximating the maximum overlap of polygons under translation,
{\em Algorithmica} {\bf 78} (2017), 147--165.

\bibitem{rogersshephard}
C.\ A.\ Rogers, G.\ C.\ Shephard,
The difference body of a convex body,
{\em Arch.\ Math.}\ {\bf 8} (1957), 220--233.

\bibitem{schneider}
R.\ Schneider,
{\em Convex Bodies: The Brunn--Minkowski Theory},
second expanded edition, Cambridge University Press, 2014.

\end{thebibliography}
\end{document}